\documentclass[12pt]{amsart}

\usepackage{amsmath,amssymb,amsthm}
\usepackage{a4wide}
\usepackage{amscd}

\newtheorem{lemma}{Lemma}[section]
\newtheorem{theorem}[lemma]{Theorem}

\newtheorem{definition}[lemma]{Definition}
\newtheorem{corollary}[lemma]{Corollary}
\newtheorem{proposition}[lemma]{Proposition}

\newcommand{\C}{\mathbb{C}}
\newcommand{\F}{\mathbb{F}}

\renewcommand{\P}{\mathbb{P}}
\newcommand{\Q}{\mathbb{Q}}
\newcommand{\R}{\mathbb{R}}

\newcommand{\Z}{\mathbb{Z}}
\newcommand{\CP}{\C\P}

\renewcommand{\o}{\otimes}

\newcommand{\Om}{\Omega}

\newcommand{\im}{\mathrm{im}}

\newcommand{\Ext}{\mathrm{Ext}}

\newcommand{\cals}{\mathcal{S}}

\newcommand{\HL}{\mathrm{H}\mathrm{L}}

\newcommand{\Ad}{\mathrm{A}\mathrm{d}}
\def\commenter#1{}

\newcommand{\depth}{\mathrm{depth}\,}
\newcommand{\ds}{\bigoplus}
\newcommand{\field}{\mathrm{F}}
\newcommand{\reals}{\R}
\newcommand{\rtnls}{\Q}

\title{String homology, and closed geodesics on manifolds which are elliptic spaces}
\author{J.D.S. Jones, J. McCleary}
\date{}

\begin{document}
\maketitle

\begin{abstract}
Let $M$ be a closed simply connected smooth manifold.  Let $\F_p$ be the finite field with $p$ elements where $p> 0$ is a prime integer.  Suppose that $M$ is an $\F_p$-elliptic space in the sense of [FHT91a].  We prove that if the cohomology algebra $H^*(M, \F_p)$ cannot be generated (as an algebra) by one element, then any Riemannian metric on $M$ has an infinite number of geometrically distinct closed geodesics.  The starting point is a classical theorem of Gromoll and Meyer [GM69].  The proof  uses string homology, in particular the spectral sequence of [CJY04], the main theorem of [McC87], and the structure theorem for elliptic Hopf algebras over $\F_p$ from [FHT91a].
\end{abstract}

\section{Introduction}
We work over a ground field $\F$ and use $\F$ as the coefficients of homology and cohomology.  Our main applications are in the case where this ground field is the finite field $\F_p$ with $p$ elements, where $p > 0$ is a prime integer.

Let $\HL_*(M)$ denote the string homology algebra of a closed, simply connected manifold $M$.  String homology is a graded commutative 
$\F-$algebra defined as follows.  Let $LM$ be the free loop space of $M$.  In [CS99], Chas and Sullivan define the {\em string product}
$$
H_p(LM) \o H_q(LM) \to H_{p+q-n}(LM)
$$
where $n$ is the dimension of $M$. This product is studied from the point of view of homotopy theory in [CJ02].  The {\em string homology algebra} is defined by setting $\HL_s(M) = H_{s+n}(LM;\F)$ and using the string product to define the product.  It is proved that this string product makes $\HL_*(M)$ into a graded commutative $\F-$algebra in both [CS99] and [CJ02]. 

Our main result about string homology is the following theorem. In the statement $\Om X$ refers to the based loop space of $X$.

\begin{theorem}
Let $M$ be a simply connected, closed manifold.  Suppose there is a constant $C$ and  an integer $K$ such that
$$
 \sum_{i \leq n} \dim H_i(\Om M; \F_p) \leq Cn^K.
$$
Let $K_0$ be the minimal exponent which can occur in this bound.  Then the string homology algebra $\HL_*(M; \F_p)$ contains a polynomial algebra $P$ over $\F_p$ on $K_0$ generators and $\HL_*(M; \F_p)$ is a finitely generated free module over $P$.  
\end{theorem}  

If $H_*(\Om M; \F_p)$ satisfies the growth hypotheses in the statement of this theorem, then we say that $H_*(\Om M; \F_p)$ has {\it polynomial growth}.  The main application of this theorem is the following result.

\begin{theorem}
Let $M$ be a simply connected, closed manifold.  Suppose $H_*(\Om M; \F_p)$ has polynomial growth and the algebra $H^*(M; \F_p)$ cannot be generated by one element.  Then for any metric on $M$, there is an infinite number of geometrically distinct closed geodesics on $M$.
\end{theorem}

To obtain this result from Theorem 1.1 we use the Gromoll - Meyer theorem relating closed geodesics and the topology of the free loop space.  A metric on $M$ defines a function, the {\em energy function}, on $LM$ given by 
$$
\gamma \mapsto \displaystyle \int_{S^1} \langle \gamma'(t), \gamma'(t)\rangle\, dt.
$$
 If $\gamma\colon S^1 \to M$ is a closed geodesic parametrized by arc length then $\gamma$ is a critical point of the energy function, as is the  loop $\gamma_n$ defined by $\gamma_n(z) = \gamma(z^n)$.  Furthermore every critical point of the energy function is of the form $\gamma_n$ where $\gamma$ is a closed geodesic parametrised by arc length [B56].

The circle $S^1$ acts on $LM$ by rotating loops and the energy function is  $S^1$-invariant.  It follows that any closed geodesic $\gamma$ parametrized by arc length generates an infinite number of critical $S^1$ orbits of the energy function.   In general these orbits will not be isolated but if there are only a finite number of geometrically distinct closed geodesics these orbits will be isolated.

We use the following terminology for graded vector spaces.  If each $V_i$ is finite dimensional we say $V$ has {\em finite type}.  If $V$ has finite type then we say it has {\em finite dimension} if $\dim V_i$ is zero for all but a finite number of $i$, {\em infinite dimension} if $\dim V_i$ is non-zero for an infinite number of $i$, and {\em doubly infinite dimension} if the sequence of numbers $\dim V_i$ is unbounded.  Note that doubly infinite dimension is the same as polynomial growth with minimal exponent at least $2$.  Using Morse-Bott theory, Gromoll and Meyer showed in [GM69] that the relation between critical points of the energy function and closed geodesics leads to the following theorem.

\begin{theorem}
Let $M$ be a simply connected closed manifold.  If $H_*(LM; \F)$ has doubly infinite dimension for some field $\F$, then for any metric on $M$ there is an infinite number of geometrically distinct closed geodesics on $M$.
 \end{theorem}
 
If $\pi_1(M)$ is finite, then we can apply this theorem to the universal cover $\tilde{M}$ of $M$.  If $\pi_1(M)$ is infinite and $\pi_1(M)$ has an infinite number of conjugacy classes, then $LM$ has an infinite number of components.    Given  a metric on $M$ we can choose a minimiser of the energy function in each component  of $LM$ and it follows that this metric has an infinite number of geodesics [BaThZ81].  This leaves the case where $\pi_1(M)$ is infinite but only has a finite number of conjugacy classes.  Very little is known about this case [BngHi84].
  
In [V-PS76] Sullivan and Vigu\'e-Poirrier took up the case where $\F = \Q$ and, as an application of the theory of minimal models in rational homotopy, proved the following theorem.  

\begin{theorem}
Suppose $M$ is a closed, simply connected manifold and the algebra $H^*(M;\Q) $ is not generated by one element.  Then $H_*(LM; \Q)$ is doubly infinite.
\end{theorem}

A key ingredient in the proof of Theorem 1.2 from Theorem 1.1  is the following theorem from [McC87].  

\begin{theorem}
Let $X$ be a simply connected space such that the algebra $H^*(X; \F_p)$ cannot be generated by one element.  Then $H_*(\Om X; \F_p)$ is doubly infinite.
\end{theorem}

The main idea which led to this paper is to use string homology with coefficients in $\F_p$ to convert Theorem 1.5 into a result about string homology. The first step in this process is to use the spectral sequence of [CJY04] to relate string homology and the homology of the based loop space.  The second is to use the structure theorems for elliptic Hopf algebras over $\F_p$ from [FHT91a] to control the input into this spectral sequence.

This paper is set out as follows.  In \S2 we deal with those aspects of string homology our main results require.  The primary objective in \S2 is to prove Theorem 2.1.  In \S3 we give applications of Theorem 2.1.  For example we explain how this theorem applies to the main examples of [McCZ].  In \S4 we summarize the results  from [FHT91a]  we need and complete the proof of the main theorems.  Finally in \S5 we give applications of the main theorem to homogeneous spaces.

Both authors would like to acknowledge the support of the Isaac Newton Institute in Cambridge during the Grothendieck-Teichm\"uller Groups, Deformation and Operads (GDO) programme in 2013 where this project began.

\section{String homology}

In [CJY04, Theorem 1], it is shown that there is a multiplicative second quadrant spectral sequence 
$(E_r^{s,t}, d_r^{s,t})$ with
$$
d_r^{s,t}\colon E_r^{s,t} \to E_r^{s-r, t+r-1}
$$$$
E_2^{s,t} = H^{-s}(M)\otimes H_t(\Om M)
$$
and converging to $\HL_*(M)$.

Here second quadrant means that $E_r^{s,t}$ is zero if $s > 0$ or $t < 0$.  Multiplicative means that each term $E_r^{s,t}$ is a bigraded algebra, $d_r^{s,t}$ is a bigraded derivation of the product, and the $E_{\infty}$ term of the spectral sequence is the bigraded algebra associated to a filtration of $\HL_*(M)$.  The edge homomorphism $h\colon \HL_*(M) \to E_{\infty}^{0,*} \subseteq H_*(\Om M)$ is the natural algebra homomorphism $h\colon \HL_*(M) \to H_*(\Om M)$.  This give us a method of relating the algebras $H_*(\Om M)$ and $\HL_*(M)$.  

The simplest way to think of this spectral sequence is to use the string topology spectrum $\cals(M) = LM^{-TM}$ introduced in [CJ04].  The skeletal filtration of $M$ induces a filtration of $LM$ using the evaluation map $LM \to M$, and this in turn induces a filtration of $\cals(M)$.  The spectral sequence is the spectral sequence obtained from this filtration of $\cals(M)$.

Our main application of this spectral sequence is the following theorem.

\begin{theorem}  
Let $M$ be a closed oriented manifold. Then $\HL_*(M; \F_p)$ contains a polynomial algebra over $\F_p$ on $k$ generators if and only if the centre of $H_*(\Om M; \F_p)$ contains a polynomial algebra over $\F_p$ on $k$ generators. 
\end{theorem}

The first step is to prove the following lemma.  

\begin{lemma}  
Let $M$ be a closed manifold. The kernel of the ring homomorphism $h\colon \HL_*(M) \to H_*(\Om M)$ is a nilpotent ideal.
\end{lemma}

\begin{proof}
Let
$$
0 = F^{-n-1} \subseteq  F^{-n} \subseteq \dots \subseteq F^0 = \HL_*(M)
$$
be the (negatively indexed) filtration of $\HL_*(M)$ coming from the CJY spectral sequence.  Here $n$ is the dimension of the manifold $M$.  Then 
$$
F^{-i}  F^{-j}\subseteq F^{-i-j}
$$
and so $(F^{-1})^{n+1} = 0$.  The proposition follows since $F^{-1}$ is exactly the kernel of the edge homomorphism of this spectral sequence.
\end{proof}

Next we give a very simple but very useful lemma.

\begin{lemma}
Suppose $M$ is a closed, simply connected manifold of dimension $n$.  Let $C$ be the centre of the algebra $H_*(\Om M; \F_p)$.  Then for any $x \in C$
$$
x^{p^{n-2}}  \in \im(h\colon \HL_*(M; \F_p) \to H_*(\Om M; \F_p)).
$$
\end{lemma}

\begin{proof}
Because $h$ is the edge homomorphism in the CJY spectral sequence we know that an element $y \in H_*(\Om M; \F_p) = E_2^{0,*}$ is in the image of $h$ if and only if it is an infinite cycle in this spectral sequence.  Let $x \in H_*(\Om M; \F_p) = E_2^{0, *}$ be a central element.   Now $x$ may or may not be a cycle for $d_2$ in the CJY spectral sequence.  But $d_2$ is a derivation and $x$ is central so it follows that
$$
d_2 x^p = px^{p-1} d_2 x.
$$
Since the ground field is $\F_p$ we have that $d_2 x^p = 0$.  It may or may not be the case that $x^p$ is a cycle for $d_3$ but the same argument shows that $x^{p^2} = (x^p)^p$ is a cycle for $d_3$. Because $M$ has dimension $n$, $d_r = 0$ for $r \geq n + 1$.  Since $M$ is simply connected $H^1(M; \F_p) = H^{n-1}(M; \F_p) = 0$.  It follows that there are at most $n-2$ differentials on $E_2^{0,*}$ that could be non-zero, starting with $d_2$.  Repeating this argument at most $n-2$ times shows that $x^{p^{n-2}}\in E_2^{0,*}$ is an infinite cycle and it follows that $x^{p^{n-2}}$ is in the image of $h$.
\end{proof} 

We will also need the following result of [FTV].

\begin{theorem}
The image of $h\colon \HL_*(M; \F_p) \to H_*(\Om M; \F_p)$ is contained in the centre of $H_*(\Om M, \F_p)$. 
\end{theorem}

To prove Theorem 2.1 we simply combine the previous three results.

\begin{proof}[Proof of Theorem 2.1]
The kernel of $h\colon \HL_*(M; \F_p) \to H_*(\Om M; \F_p)$ is a nilpotent ideal, and the image of $h$ is contained in the centre of $H_*(\Om M; \F_p)$.  So if $\HL_*(M; \F_p)$ contains a polynomial algebra on $k$ generators, then so does the centre of $H_*(\Om M; \F_p)$.  On the other hand, if the centre of $H_*(\Om M; \F_p)$ contains the polynomial algebra $\F_p[x_1, \dots , x_k]$, then Lemma 2.3 shows that every element of the sub-algebra of the $E_2$-term of the CJY spectral sequence
$$
\F_p[(x_1)^{p^{n-2}}, \dots, (x_k)^{p^{n-2}}] \subset H_*(\Om M; \F_p) = E_2^{0, *}
$$
is an infinite cycle.   It follows that $\HL_*(M; \F_p)$ contains a polynomial algebra on $k$ generators.
\end{proof}

\section{Applications of Theorem 2.1}

\subsection{Sphere bundles over spheres}
Let $Q = Q_{2n, e}$ denote the sphere bundle
$$
S^{2n-1} \to Q \to S^{2n}
$$
with Euler class $e \in \Z$.   We choose an orientation of $S^{2n}$ to identify the Euler class with an integer.  We prove the following result.

\begin{proposition}
If $e \neq 0$, for any metric on $Q = Q_{2n, e}$, there is an infinite number of closed geodesics on $Q$.
\end{proposition}

Notice that a sphere bundle over a sphere $M$ not of the form $Q_{2n,e}$ with $e \neq 0$ has its rational cohomology ring generated by more than one element.  Therefore the theorem of Sullivan and Vigu\'e - Poirrier, Theorem 1.4 shows that  any metric on $M$ has an infinite number of closed geodesics.  However, $Q_{2n,e}$ is a rational homology sphere if $e \neq 0$.

\begin{proof}[Proof of Proposition 3.1]
Choose a prime $p$ such that $p$ divides $e$.  Standard basic calculations in algebraic topology show that 
$$
H^*(Q; \F_p) = E[a_{2n-1}, b_{2n}], \mbox{\rm\ and\ }
H_*(\Om Q; \F_p) = P[u_{2n-2}, v_{2n-1}].
$$
Here $E$ denotes the exterior algebra over $\F_p$ and $P$ the polynomial algebra over $\F_p$.  The subscripts are the degrees of the elements.  If $p = 2$, then the algebra $P[u_{2n-2}, v_{2n-1}]$ is not graded commutative since $v_{2n-1}^2 \neq 0$.  However the centre of $H_*(\Om Q; \F_p)$ is precisely $P[u_{2n-2}, v_{2n-1}^2]$.  Theorem 2.1 shows that $\HL_*(Q)$ contains a polynomial algebra on two generators and so $H_*(LQ; \F_p)$ has doubly infinite dimension.  The Gromoll-Meyer theorem shows that for any metric on $Q$, there is an infinite number of distinct closed geodesics.  
\end{proof}

\subsection{The Grassmannian of oriented two planes in $\R^{2n+1}$}
Let $G_2^+(\R^{2n+1})$ denote the Grassmannian of oriented $2$-planes in $\R^{2n+1}$.  Recall the following two calculations from the theory of characteristic classes.
\begin{enumerate}
\item
Suppose $2$ is a unit in the coefficient field  $\F$.  Then
$$
H^*(G_2^+(\R^{2n+1}); \F) = P[x_2]/(x_2^{2n}),
$$
\item
\qquad \qquad \qquad \qquad \  \ \,$
H^*(G_2^+(\R^{2n+1}); \F_2) = P[x_2]/(x_2^{n})\o E(y_{2n}).
$
\end{enumerate}
So the algebra $H^*(G_2^+(\R^{2n+1}); \F_p)$ can be generated by a single generator for $p\neq 2$, but in the
case $p = 2$ it requires at least two generators.
Another standard calculation in algebraic topology shows that
$$
H_*(\Om G_2^+(\R^{2n+1}); \F_2) = E(u_1) \o P[v_{2n-2}] \o P[w_{2n-1}] \cong H_*(\Om(\CP^n \times S^{2n}); \F_2).
$$
Evidently this contains a central polynomial algebra generated by two elements.  The following theorem 
follows from the Gromoll - Meyer theorem in the case of $\F_2$  coefficients.

\begin{theorem}
Any metric on $G_2^+(\R^{2n+1})$ has an infinite number of closed geodesics.
\end{theorem}

\subsection{The list of examples from [McCZ]}
There is a list in [McCZ], based on the work of [O63], 
consisting of one representative from each diffeomorphism class of homogeneous spaces $G/K$, where $G$ is a compact connected Lie group and $K$ is a connected closed subgroup, with two properties:
\begin{itemize}
\item
$G/K$ is not diffeomorphic to a sphere,  a real, complex, or quaternionic projective space, nor is it diffeomorphic to the Cayley projective plane.
\item
The algebra $H^*(G/K; \Q)$ is generated by one element;
\end{itemize}
In other words it is the list of examples of homogeneous spaces to which we we would like to apply the theorem of Gromoll - Meyer, but cannot do so over the ground field $\Q$.
This list contains two infinite families. 
\begin{itemize}
\item
The Stiefel manifold $V_2(\R^{2n+1})$ of two frames in $\R^{2n+1}$.  This is a $2n-1$ sphere bundle over $S^{2n}$ with Euler class $2$, and Proposition 3.1 shows  that any metric on $V_2(\R^{2n+1})$ has an infinite number of geometrically distinct closed geodesics. 
\item
The Grassmannian of oriented $2$-planes in $\R^{2n+1}$.  Theorem 3.2 shows that any metric on this manifold has an infinite number of geometrically distinct closed geodesics.
\end{itemize}

There are another $7$ homogeneous spaces on this list.  The first two are $SU(2)/SO(3)$ and $Sp(2)/SU(2)$, and the other $5$ are homogeneous spaces for $G_2$.  It is possible to go through these $7$ examples by direct calculations with loop spaces.  However we will deal with them in \S 5 as examples of our main theorem.

\def\Hopfquot{/\hskip -3pt/}

\section{The proofs of Theorem 1.1 and Theorem 1.2}

We next need results contained in a series of inter-related papers by F\'elix, Halperin, Lemaire, and Thomas on the homology of based loop spaces.  We give a brief summary of the results we need.

\subsection{Elliptic Hopf algebras}
Let $\Gamma$ be a graded Hopf algebra over the ground field $\F$.  The {\em lower central series} of  $\Gamma$ is the sequence 
$$
\Gamma = \Gamma^{(0)}  \supset \Gamma^{(1)} \supset \Gamma^{(2)} \supset \cdots \supset \Gamma^{(n)} \supset \cdots
$$
where
$\Gamma^{(i+1)} = [\Gamma, \Gamma^{(i)}]$. 
By definition $\Gamma$ is {\em nilpotent} if $\Gamma^{(s)} = \F$ for some $s$.  Although the definition of the $\Gamma^{(i)}$ depends only on the algebra structure of $\Gamma$, it is straightforward to check that the $\Gamma^{(i)}$ are normal sub Hopf algebras of $\Gamma$.

We say that $\Gamma$ is {\em connected} if $\Gamma_i = 0$ when $i < 0$ and $\Gamma_0 = \F$, and that $\Gamma$ is {\em finitely generated} if it is finitely generated as an algebra.  From  [FHT91a] we have the following definition.

\begin{definition}  
Fix a ground field $\F$.  A Hopf algebra $\Gamma$ over $\F$ is {\bf elliptic} if it is connected, co-commutative, finitely generated, and nilpotent.
\end{definition}

Note that the only part of the definition of an elliptic Hopf algebra which refers to the coproduct is the condition that it is co-commutative.  

Here are some examples.  In these examples we assume that the Hopf algebras in question are connected and co-commutative over a fixed ground field $\F$.
\begin{enumerate}
\item
If $\Gamma$ is a finite dimensional Hopf algebra, then $\Gamma$ is elliptic.  To prove this first note that since $\Gamma$ is connected $\Gamma^{(i)}$ is $i+1$ connected.  Since $\Gamma$ is finite dimensional it  follows that $\Gamma^{(i)} = \F$ for sufficiently large  $i$.  So $\Gamma$ is nilpotent.  Since $\Gamma$ is finite, it is finitely generated.
\item
If $\Gamma$ is commutative, then $\Gamma$ is elliptic if and only if $\Gamma$ is finitely generated.
\item 
Let $L$ be a  Lie algebra.  Let $U(L)$ be the universal enveloping algebra of $L$.  This becomes a Hopf algebra by defining the coproduct to be the unique coproduct which makes the elements of $L$ primitive.  Then $U(L)$ is an elliptic Hopf algebra if and only if $L$ is a finitely generated nilpotent Lie algebra.
\end{enumerate}

The structure theorem for elliptic Hopf algebras proved in  [FHT91a]  tells us that essentially these examples generate the class of all elliptic Hopf algebras by taking extensions.

\begin{theorem}
Let $\F$ be a field and let $\Gamma$ be a connected, finitely generated, co-commutative Hopf algebra over $\F$. 
\begin{itemize}
\item  
If $\F$ has characteristic zero, then $\Gamma$ is elliptic if and only if $\Gamma = U(L)$ where $L$  is a finitely generated, nilpotent Lie algebra over $\F$.
\item
If $\F$ has characteristic $p \neq 0$, then $\Gamma$ is elliptic if and only if  it contains a finitely generated, central sub Hopf algebra $C$, such that $\Gamma\Hopfquot C$ is a finite dimensional algebra. 
\end{itemize}
\end{theorem}
The statement of the second clause of the theorem is not quite the same as the statement (ii) in Theorem~B of [FHT91a] but it is easily seen to be equivalent to it.  From [MM65] we know that $\Gamma$ is isomorphic to $C \o \Gamma\Hopfquot C$ as a $C$ algebra.  Since $C$ is finitely generated and commutative it follows from a theorem of Borel [MM65] that as an algebra $C$ is isomorphic to $P\o A$, where $P$ is a polynomial algebra over $\F$ in a finite number of variables and $A$ is a finite dimensional algebra. It follows that $\Gamma$ is isomorphic to $P \o A \o \Gamma\Hopfquot C$ as a $P$ module.  Since both $A$ and $\Gamma\Hopfquot C$ are finite dimensional it follows that $\Gamma$ is a finitely generated free module over $P$.  This is the condition given in [FHT91a].

\subsection{Depth and the Gorenstein condition}

Let $A$ be a graded augmented algebra over the ground field $\F$.  We will assume that $A$ is connected.  We can form the  vector spaces
$$
\Ext^{i,j}_A(\F, A).
$$
The {\em depth} of $A$,  $\depth A$, is defined as follows:
$$
\depth A = \inf \{s \mid \Ext^{s, *}_A(F, A) \neq 0 \}.
$$
If $n = \depth A$, then $\Ext ^{s,t}_A(\F, A) = 0$ for $s < n$ and there is an integer $t$ such that $\Ext^{n, t}_A(\F,A) \neq 0$.  In particular the depth of $A$ could be infinite, and this means that $\Ext ^{s,t}_A(\F, A) = 0$ for all $(s,t)$.

The graded algebra $A$ is {\em Gorenstein} if there is a pair of integers $(n,m)$ such that 
\begin{itemize}
\item 
$\Ext_A^{s,t}(\F,A) = 0$ if $(s,t) \neq (n,m)$,
\item
$\Ext_A^{n,m}(\F,A) = \F$.
\end{itemize}
The definition of depth and the Gorenstein condition first appear in classical commutative ring theory.  

It is straightforward to check that,
\begin{itemize}
\item
$\depth A\o B = \depth A + \depth B$,
\item
$A \o B$ is Gorenstein if and only if both $A$ and $B$ are Gorenstein.
\end{itemize}

In the case of a polynomial algebra $\F[x]$ with one generator of degree $k$
$$
\Ext^{1,k}_{\F[x]}(k, \F[x]) = \F, \quad \Ext^{s,t}_{\F[x]}(\F, \F[x]) = 0  \quad (s,t) \neq (1, k). 
$$
In the case where $A = \F[x]/(x^n)$ is a truncated polynomial algebra with generator of degree $k$ then
$$
\Ext^{0, -k(n-1)}_A(\F, A) = \F, \quad \Ext^{s,t}_A(\F, A) = 0 \quad (s,t) \neq (0, -k(n-1)).
$$
The most elementary method for doing these calculations is to use the minimal resolution of $\F$ over $\F[x]$ and the minimal resolution of $\F$ over $\F[x]/(x^n)$.   
It follows that both the algebras $\F[x]$ and $\F[x]/(x^n)$ are Gorenstein, and  
$$
\depth \F[x] = 1, \quad \depth \F[x]/(x^n) = 0.
$$

The following lemma is contained in [FHT88] (see Proposition~1.7).

\begin{lemma}
Suppose $A$ is an infinite tensor product of algebras.  Then the depth of $A$ is infinite.
\end{lemma}

Suppose $\Gamma$ is a connected Hopf algebra that is commutative as an algebra.  By a theorem of Borel [MM65, Theorem  7.11], it follows that 
$\Gamma$ is isomorphic as an algebra to a tensor product of polynomial algebras and truncated polynomial algebras.  If $\Gamma$ is not finitely generated then Lemma 4.3 shows that $\Gamma$ has infinite depth.  If $\Gamma$ is finitely generated, then it has finite depth and it is isomorphic to $P \o A$ where $P$ is a polynomial algebra with $m = \depth \Gamma$ variables and $A$ is a finite tensor product of truncated polynomial algebras.  This proves Theorem 4.2 in the case where $\Gamma$ is  commutative.  One way to think of the proof of 4.2 is that it works by reducing the general case to the commutative case by using the condition that $\Gamma$ is nilpotent.

The relevance of depth and Gorenstein to topology comes from results in [FHLT89] and [FHT88], which we state as the following theorem.

\begin{theorem}
Let $X$ be a simply connected finite complex.
\begin{enumerate}
\item
The Hopf algebra $H_*(\Om X;\F)$ has finite depth. In fact, $\depth X \leq \mbox{\rm LSCat}\, X$, where $\mbox{\rm LSCat}\, X$ denotes
the Lyusternik-Schnirelman category of $X$.
\item
If the Hopf algebra $H_*(\Om X;\F)$ is Gorenstein, then $X$ is a Poincar\'e duality space.
\end{enumerate}
\end{theorem}
In [FHT88] F\'elix, Halperin and Thomas extend the Gorenstein condition to differential graded algebras and show that a finite complex $X$ is a Poincar\'e duality space if and only if the cochain algebra $S^*(X;\F)$ is a Gorenstein differential graded algebra.   While it is true that if $H^*(X;\F)$ is Gorenstein then so is $S^*(X;\F)$  the reverse implication is not true -- see Examples 3.3 of [FHT88].

If $X$ is a finite complex, then we know that $H_*(\Om X;\F)$ has finite type and finite depth.  The following theorem gives some useful practical ways to deduce, in addition, that $H_*(\Om X;\F)$ is elliptic.  For the proof see Theorem C of [FHT91a].

\begin{theorem}
Suppose $\Gamma$ is a connected, cocommutative Hopf algebra over $\F$ of finite type and $\Gamma$ has finite depth.  Then the following are equivalent:
\begin{enumerate}
\item $\Gamma$ is elliptic,
\item $\Gamma$ is nilpotent,
\item $\Gamma$ has polynomial growth,
\item $\Gamma$ is Gorenstein.
\end{enumerate}
\end{theorem}

\subsection{The proof of Theorem 1.1.}
If $M$ is a closed, connected, oriented manifold of finite dimension, then $H_*(\Om M; \F_p)$ is connected and cocommutative, and it has finite type and finite depth.  We are assuming it has polynomial growth.  It follows from Theorem 4.5 that $H_*(\Om M; \F_p)$ is elliptic.  Therefore, from Theorem 4.2, it is a finitely generated free module over a central sub-algebra $P$ that is a polynomial algebra on a finite number, say $l$, variables.  It follows that $H_*(\Om M; \F_p)$ has polynomial growth with exponent $l$ and indeed $l$ is the minimal exponent which can occur in the inequality for polynomial growth.  In the notation of Theorem 1.1, $l = K_0$.  This proves Theorem 1.1.

\subsection{The proof of Theorem 1.2.} 
It follows from Theorem 4.2 that if $\Gamma$ is an elliptic Hopf algebra over $\F_p$, then $\Gamma$ is  doubly infinite if and only if the centre of $\Gamma$ contains a polynomial algebra on two generators.
Now let $M$ be a simply connected closed manifold satisfying the hypotheses of Theorem 1.2.  Then, as in the proof of Theorem 1.1, it follows that $H_*(\Om M; \F_p)$ is an elliptic Hopf algebra.  Suppose in addition that  the algebra ${H}^*(M; \F_p)$ cannot be generated by one element.   From Theorem 1.5, it follows that $H_*(\Om M; \F_p)$ is doubly infinite and so the centre of $H_*(\Om M; \F_p)$ contains a polynomial algebra on two generators.  By theorem 2.1 it follows that $\HL_*(M; \F_p)$ contains a polynomial algebra on two generators and therefore $H_*(LM, \F_p)$ is doubly infinite.  The Gromoll - Meyer theorem, Theorem 1.3, completes the proof.

\section{Application to homogeneous spaces.}

The following theorem is Example 3.2 in [FHT93]:

\begin{theorem}
Let $G$ be a simply connected, compact Lie group and $K$, a connected, closed subgroup of $G$.  
Then the homogeneous space $G/K$ is $\F_p$ elliptic for any prime $p$.
\end{theorem}

%Yes: See 3.2 in [FHT93]: {\em John:  Have I  got the connectivity, simple connectivity hypotheses correct?}

The proof uses the fibration
$$
\Om G \to \Om (G/K) \to K
$$
for which
the fundamental group $\pi_1(K)$ acts trivially on the groups $H_*(\Om G; F_p)$.
Then a Leray-Serre spectral sequence argument may be applied because $K$ and $\Om G$ are both elliptic and
hence have polynomial growth.

%This follows because the analogous result is true in the following generality.  
%Let $X \to E \to B$ be a fibration and consider the fibration $\Om X \to 
%\Om E \to \Om B$.  Then $\pi_1(\Om B)$ acts trivially on $H_*(\Om X)$.

%Since the looped fibration is a fibration of H-spaces, you are right. See Lemma 5.1 of Browder's on differential Hopf algebras.
% {\em John: Have I just made this up?}

\commenter{
\begin{lemma}
The algebra $H_*(\Om G; \F_p)$ is a finite tensor product of polynomial algebras and truncated polynomial algebras.
\end{lemma}
}
\commenter{
Given these two lemmas the proof of the theorem is a simple argument with the Serre spectral sequence in 
homology of the fibration $\Om G \to \Om G/K \to K$.  The $E_2$ term of this spectral sequence is
$$
H_*(K; \F_p) \o H_*(\Om G; \F_p).
$$
Now $H_*(K; \F_p)$ is finite $H_*(\Om G; \F_p)$ has polynomial growth and the Serre spectral sequence shows that $H_*(\Om G/K; \F_p)$ has polynomial growth.  Since $G/K$ is a compact manifold $H_*(G/K; \F_p)$ has finite type and finite depth and it follows that $G/K$ is $\F_p$ elliptic. 
}

Now return to the list from [McCZ].  The $7$ examples of homogeneous spaces in this list not covered by Theorems 3.1 and 3.2 are $\F_p$ elliptic spaces for any prime $p$ by Theorem 5.1.   Furthermore, in each case, there is a prime $p$ such that the cohomology algebra of the homogeneous space cannot be generated by a single element.  Therefore by Theorem 1.2 any metric has an infinite number of geometrically distinct closed geodesics.

\bigskip
\centerline{\textsc{References}}

\medskip\noindent
[BaThZ81], W.~Ballman, G.~Thorbergsson, W.~Ziller, Closed geodesies and the fundamental group, Duke
Math. J. {\bf 48}(1981), 585--588.

\smallskip\noindent
[BngHi84] V.~Bangert, N.~Hingston, Closed geodesics on manifolds with infinite abelian fundamental group,
J. Differential Geom. {\bf 19}(1984), 277--282.

\smallskip\noindent
[B56], R.~Bott, On the iteration of closed geodesies and the Sturm intersection theory,
Comm.~Pure Appl.~Math. {\bf 9}(1956), 171--206.

\smallskip\noindent
[CS99] M.~Chas, D.~Sullivan, String Topology, preprint: {\tt math.GT/9911159}, 1999.

\smallskip\noindent
[CJ02] R.L.~Cohen, J.D.S.~Jones, A Homotopy Theoretic Realization Of String Topology, Math.~Ann,
{\bf 324} (2002), 773--798. Progr.~Math., 215, Birkh\"auser, Basel, 2004.

\smallskip\noindent
[CJY04] R.L.~Cohen, J.D.S.~Jones, J.~Yan, The loop homology algebra of spheres and projective spaces,
Categorical decomposition techniques in algebraic topology (Isle of Skye 2001)
Prog.~Math {\bf 215} (2004), 77--92.

\smallskip\noindent
[F00] Y.~F\'elix, Croissance polynomiale de certains espaces de lacets. 
Acad. Roy. Belg. Cl. Sci. M\'em. Collect. 8o (3) 17 (2000), 79 pp.

\smallskip\noindent
[FHT88] Y.~F\'elix, S.~Halperin, J.C.~Thomas, Gorenstein spaces, Adv. Math. {\bf 71}(1988), 92--112.

\smallskip\noindent
[FHLT89] Y.~F\'elix, S.~Halperin, J.-M.~Lemaire, J.-C.~Thomas, Mod p loop space homology, Inv.~Math. {\bf 95}(1989),
247--262.

\smallskip\noindent
[FHT91a]  Y.~F\'elix, S.~Halperin, J.C.~Thomas, Elliptic Hopf algebras, J.~London Math.~Soc. (2) {\bf 43}(1991), 545--555.

\smallskip\noindent
[FHT91b] Y.~F\'elix, S.~Halperin, J.-C.~Thomas, Elliptic spaces, Bull.~Amer.~Math.~Soc. {\bf 25}(1991), 69--73.

\smallskip\noindent
[FHT93] Y.~F\'elix, S.~Halperin, J.-C.~Thomas, Elliptic spaces. II, Enseign.~Math. {\bf 39}(1993), 25--32.

\smallskip\noindent
[FTV] Y.~F\'elix, J.C.~Thomas, M.~Vigu\'e-Poirrier, Rational string topology,
J.~Eur.~Math.~Soc. {\bf 9}(2007), 123--156.

\smallskip\noindent
[GM69] D.~Gromoll, W.~Meyer, Periodic geodesics on compact Riemannian manifolds, 
J.~Diff.~Geom. {\bf 3} (1969), 493--510.

\smallskip\noindent
[L11], F.~Laudenbach,  A Note on the Chas-Sullivan product. L'Enseignement Math. {\bf 57}(2011) 1--19.

\smallskip\noindent
[McC87] J.~McCleary, On the mod $p$ Betti numbers of loop spaces, Invent.~Math. {87} (1987), 643--654.

\smallskip\noindent
[McCZ] J.~McCleary, W.~Ziller, On the free loop space of homogeneous spaces. Amer.~J.~Math. {\bf 109}(1987), 765--781.
Corrections to: ``On the free loop space of homogeneous spaces'' Amer.~J.~Math. {\bf 113}(1991), 375--377.

\smallskip\noindent
[MM65] J.~W.~Milnor, J.~C.~Moore, On the structure of Hopf algebras, Ann.~of Math. (2) {\bf 81}(1965), 211--264.

\smallskip\noindent
[O66], A.~L.~Oni\v{s}\v{c}ik, Transitive compact transformation groups, Mat.~Sb. {\bf 60}(1963), 447--485 [Russian].  

\smallskip\noindent
[Sp67], Spivak, M., Spaces satisfying Poincar\'e duality. Topology {\bf 6}(1967) 77--101.

\smallskip\noindent
[V-PS76] M.~Vigu\'e-Poirrier, D.~Sullivan, The homology theory of the closed geodesic problem, 
J.~Diff.~Geom. {\bf 11} (1976), 633--644.

\end{document}